\documentclass[letterpaper,11pt]{paper}
\usepackage[a4paper,bindingoffset=0.2in,%
            left=1.25in,right=1.25in,top=1in,bottom=1in,%
            footskip=.25in]{geometry}
\usepackage{amssymb}
\usepackage{amsmath}
\usepackage{amsthm}
\usepackage{graphicx}
\usepackage{multirow}
\usepackage{algorithmic}
\usepackage{algorithm}
\usepackage{tikz}
\usetikzlibrary{arrows}
\tikzstyle{block}=[draw opacity=0.7,line width=1.4cm]

\usepackage{tkz-graph}

\tikzset{LabelStyle/.style= {draw,
fill = yellow,
text = red}}

\usepackage{tkz-berge}

\newtheorem{thm}{Theorem}[section]

\newtheorem{teor}[thm]{Theorem}

\newtheorem{defn}[thm]{Definition}

\begin{document}
\thispagestyle{empty}
\begin{center}
\LARGE{\textrm{{
\fontsize{14pt}{\baselineskip}\selectfont \textbf{A novel low-rank matrix completion approach to estimate missing entries in Euclidean distance matrices}}
}}
\end{center}
\vspace{0.5cm}
\begin{center}
$\bullet$ N. Moreira$^\dag$  $\quad\bullet$ L. T. Duarte$^\S$ $\quad\bullet$  C. Lavor$^\dag$ $\quad\bullet$  C. Torezzan$^\S$
\end{center}
\vspace{0.5cm}
$\dag$ {\scriptsize{Department of Applied Mathematics, University of Campinas, Campinas - SP - Brasil, 13083-970}}\\
$\S$ {\scriptsize{School of Applied Sciences, University of Campinas, Limeira - SP - Brasil, 13484-350}}

\vspace{0.5cm}

\begin{abstract}
A Euclidean Distance Matrix (EDM) is a table of distance-square between points on a $k$-dimensional Euclidean space, with applications in many fields (e.g. engineering, geodesy, economics, genetics, biochemistry, psychology). A problem that often arises is the absence (or uncertainty) of some EDM elements. In many situations,  only a subset of all pairwise distances is available and it is desired to have some procedure to estimate the missing distances. In this paper, we address the problem of missing data in EDM through low-rank matrix completion techniques. We exploit the fact that the rank of a EDM is at most $k+2$ and does not depend on the number of points, which is, in general, much bigger then $k$. We use a Singular Value Decomposition approach that considers the rank of the matrix to be completed and computes, in each iteration, a parameter that controls the convergence of the method. After performing a number of computational experiments, we could observe that our proposal was able to recover, with high precision, random EDMs with more than one thousand points and up to 98 percent of missing data in few minutes. Additionally, our method required a smaller number of iterations when compared to other competitive state-of-art technique.
\end{abstract}



\vspace{1.5cm}

\section{Introduction}

Euclidean distances are found in many real-world applications including fields from engineering to psychology \cite{journal:Leeuw}. For example, in sensor network localization, some pairs of sensors can communicate each other by sending and receiving  signals, which allows them to estimate their distances \cite{journal:Biswas}. Another example is the molecular conformation problem, in which some pairs of atoms are connected and the distances between them can be estimated using nuclear magnetic resonance experiments \cite{journal:Havel}. Other interesting examples and references can be found in \cite{journal:Dokmanic, journal:Liberti, journal:Billinge}.  \\
\indent Due to several reasons, in many situations we cannot access all the pairwise distances and, therefore, it becomes necessary to estimate those missing data. This problem is known as the \emph{Matrix Completion Problem} \cite{inbook:Candes1}.\\
\indent In many applications, the matrices to be completed are low-rank ones, that is, the minimum between the number of rows and the number of columns is much larger than the rank of the matrix. For instance, in molecular conformation problem, we may be interested in the coordinates of thousands of atoms (points in $\mathbb{R}^3$). This leads us to a Euclidean distance matrix with rank at most $5$, which is very low comparing to the number points \cite{journal:Dokmanic}.\\
\indent The \emph{Matrix Completion Problem} has attracted the attention in many areas of science and several advances have been obtained. One of the most important results in this field was achieved by Cand\`es and Recht \cite{inbook:Candes1}, who proved that it is possible to perfectly recover most low-rank matrices from what appears to be an incomplete set of entries.\\
\indent Based on this result, some algorithms have been proposed to recover a low-rank matrix. Cai et al \cite{journal:Cai} proposed a first-order singular value thresholding algorithm addressing problems in which the optimal solution has low rank. Mazumder et al \cite{journal:Mazumder} presented another algorithm, called Soft-Impute, which uses the singular value decomposition to find successive approximations through a regulator parameter. Krislock \cite{PhDthesis:Krislock} considers semidefinite facial reduction to complete low-rank Euclidean distance matrices.\\
\indent In this paper, we propose a novel completion approach. The rationale behind our method, which bears a resemblance to the soft-impute algorithm, is that the rank of a EDM, obtained from points in $\mathbb{R}^k$, is at most $k + 2$ and does not depend on the number of points. Therefore, differently from the soft-impute algorithm, our proposal takes the rank of the target matrix into account in order to calculate, for each iteration, the regulator parameter, which is paramount for the algorithm convergence. More specifically, we propose a way to calculate the regulator parameter, for each iteration, based on the most recent approximation. This apparently naive modification has been proved to be highly efficient, allowing to recover incomplete EDMs with high precision. For example, we were able to recover random EDMs with more than one thousand rows and up to $90\%$ of missing data, obtaining absolute error less than $10^{-8}$.  We have performed massive computational experiments that confirm the efficiency of our approach to recover missing data in EDMs, besides requiring a smaller number of iterations when compared to other competitive state-of-art technique, as the soft-impute.

\indent The rest of this paper is organized as follows. In Section $2$, we set the notation and provide a brief background on EDMs. In Section $3$, we introduce the matrix completion problem and present the soft-impute algorithm proposed by Mazumder et al \cite{journal:Mazumder}. In the same section, we present our contributions, defining a new way to choose the regulator parameter of soft-impute. In Section $4$, we present and discuss numerical experiments. Finally, we conclude the paper in Section $5$.

\section{Notations and preliminaries}

In this section, we shall define our notations and also present a short review of some important linear algebra results that are central for this paper.

Let $X=\left[x_1\,\,x_2\,\, \ldots \,\,x_n \right],\,\,\,x_i\in\mathbb{R}^k$, be a matrix whose columns represent $n$ points in a Euclidean $k-$dimensional space. The  distance-square $\left(d_{ij}\right)$ between any two points in $\mathbb{R}^k$ is defined by:
{\small{
\begin{equation}
\label{dist}
  d_{ij}=\|x_i-x_j\|^2_2 \triangleq \langle x_i-x_j\,,\,x_i-x_j\rangle = x_i^Tx_i -2x_i^Tx_j +x_j^Tx_j = \|x_i\|_2^2+\|x_j\|_2^2-2x_i^Tx_j.
\end{equation}
}}
\begin{defn}
The matrix $D=\left[d_{ij}\right]\in \mathbb{R}^{n\times n}$, whose elements represent the square of the distances between $n$ points in $\mathbb{R}^k$ is called \emph{Euclidean Distance Matrix}, denoted by $edm(X)$.
\end{defn}

\begin{defn}
The matrix $G=X^TX\in\mathbb{R}^{n\times n}$, whose elements represent the inner products between $n$ points in $\mathbb{R}^k$ is called \emph{Gram matrix} associated to $X$.
\end{defn}
According to (\ref{dist}), we can show that
\begin{equation}\label{edm}
  D \triangleq \textbf{1}\left(diag\left(X^TX\right)\right)^T-2X^TX+diag\left(X^TX\right)\textbf{1}^T,
\end{equation}
where \emph{\textbf{1}} represents the column vector of all elements equal to 1 and \emph{diag(A)} represents the column vector of the diagonal elements of $A$.

\begin{teor}\label{rankG}
  The rank of the Gram matrix associated to $X$ is at most $k$.
\end{teor}

\begin{proof}
  Since $X$ has $k$ rows, we know that $rank(X)\leq k$. From the properties
  \begin{itemize}
    \item $rank(AB)\leq\min(rank(A),\,rank(B))$
  \end{itemize}
  and
  \begin{itemize}
    \item $rank(A)=rank(A^T)$,
  \end{itemize}
 we concluded that
  $$rank(X^TX)=rank(X^TX)\leq\min(rank(X^T),\,rank(X))=rank(X)\leq k.$$
\end{proof}

\begin{teor}\label{rankD}
  The rank of the $edm(X)$ is at most $k+2$.
\end{teor}

\begin{proof}
   By Theorem \ref{rankG}, the rank of $X^TX$ is at most $k$. On the other hand, either $\textbf{1}diag(X^TX)^T$ or $diag(X^TX)\textbf{1}^T$ are matrices with rank equal to $1$.\\
   Therefore, from $rank(A+B)\leq rank(A)+rank(B)$ and by (\ref{edm}), we obtain that: $$rank(edm(X))\leq1+k+1=k+2.$$
\end{proof}

The Singular Value Decomposition (SVD) of a matrix $A$ is an important result for this work and it may be stated as the following \cite{inbook:Golub}:

\begin{teor}[Singular Value Decomposition]
Let $A\in\mathbb{R}^{n\times m}$ be a matrix with rank $r$. Then, $A$ can be written in the form
\begin{equation}\label{svd}
  A = U\Sigma V^T,
\end{equation}
where $U\in\mathbb{R}^{n\times n}$ and $V\in\mathbb{R}^{m\times m}$ are orthogonal matrices and $\Sigma\in\mathbb{R}^{n\times m}$ is given by
\begin{equation}\label{Sigma}
  \Sigma=\left[
           \begin{array}{ll}
             \Sigma_1 & 0_{r\times(m-r)} \\
             0_{(n-r)\times r} & 0_{(n-r)\times(m-r)} \\
           \end{array}
         \right],
\end{equation}
with $\Sigma_1=diag\{\sigma_1,\ldots,\sigma_r\}$.
\end{teor}
The decomposition (\ref{svd}) is called \emph{singular value decomposition} of $A$. The columns of $U$ are orthonormal  vectors, called \emph{right singular vectors} of $A$, and the columns of $V$ are also orthonormal, called \emph{left singular vectors} of $A$. $\sigma_k$ is called \emph{singular value}.

\section{Problem description and proposed model}

\subsection{Low-rank matrix completion}

Let $D$ be a Euclidean Distance Matrix (EDM), for which only a subset of its entries are available. Is it possible to guess the missing positions? How many entries are necessary to recover exactly the EDM? Is there an efficient way to compute the missing distances? This paper aims to answer these questions.

In general, the problem of estimating missing data matrix is known as the \textit{matrix completion problem} and has been intensively studied in the area of signal processing \cite{inbook:Candes1} and, more recently, attracted the attention in many other fields \cite{inbook:Candes2}.


In the context of Euclidean Distances, the matrix completion problem can be formulated as follows. Let $D$ be a EDM associated to a set of points $X = \{x_1,\,x_2,\ldots,\,x_n\}$, where $x_i\in\mathbb{R}^k$. Suppose that only a subset of its entries are available, say $\{ d_{ij}: (i,j) \in \Omega \}$, where $\Omega$  is a subset of the complete set of entries of $D$. The goal is to determine a complete EDM, $\hat{D} = [\hat{d}_{ij}]$,  such that  $\hat{d}_{ij}=d_{ij}, \ \forall \ (i,j) \in \Omega$. Additionally, the elements of the recovered matrix $\hat{D}$ must also satisfy the requirements of a EDM (non-negativity, symmetry, and triangular inequality).

The unknown elements in matrix $\hat{D}$ may be considered as variables of an optimization problem, defined by
\begin{equation}\label{minank}
    \begin{array}{ll}
      \emph{minimize} & \emph{rank($\hat{D}$)} \\
      \emph{subject to}\,\,\,   & \sum\limits_{(i,\,j)\in\Omega}(d_{ij}-\hat{d}_{ij})^2\leq\delta, \\
    \end{array}
\end{equation}
\noindent where $\delta$ is a tolerance for the known positions.\\

Unfortunately, the formulation expressed in (\ref{minank}) is not practical since the problem is non-convex and the solution through exact methods is computationally hard \cite{journal:Vandenberghe}. To overcome this difficulty, some relaxed versions of the problem  have been proposed \cite{inbook:Candes1, PhDthesis:Fazel} exploring the nuclear norm of $\hat{D}$. One of this approach is the following formulation:

\begin{equation}\label{minnuclear}
    \begin{array}{ll}
      \emph{minimize} & \|\hat{D}\|_* \\
      \emph{subject to}\,\,\,  & \sum\limits_{(i,\,j)\in\Omega}(d_{ij}-\hat{d}_{ij})^2\leq\delta, \\
    \end{array}
\end{equation}

\noindent where $\|\hat{D}\|_*$ is the nuclear norm of $\hat{D}$ (the sum of the singular values of $\hat{D}$).\\

Mazumder et al \cite{journal:Mazumder} consider an equivalent version of the problem (\ref{minnuclear}) in the Lagrangian form:
\begin{equation}\label{Lagrange}
    \begin{array}{cl}
      \emph{minimize} & \dfrac{1}{2}\sum\limits_{(i,\,j)\in\Omega}(d_{ij}-\hat{d}_{ij})^2+\lambda\|\hat{D}\|_*,
    \end{array}
\end{equation}

\noindent where $\lambda$ is a regularization parameter that controls the nuclear norm of the minimizer $\widehat{X}_\lambda$ of (\ref{Lagrange}). The authors show that there is a binary relation between $\delta\geq0$  and $\lambda\geq0$ on the active domains of each of these parameters.\\
\indent In the same paper, \cite{journal:Mazumder}, it is proposed a heuristic approach to solve the problem. The authors present a method called \emph{Soft-impute} that exploits the SVD factorization of the incomplete matrix. A pseudo-code is presented in  Algorithm $1$, where $S_{\lambda}(X)=U\Sigma_\lambda V^T,$
 with $\Sigma_\lambda=diag[(\sigma_1-\lambda)_+,\ldots,\,(\sigma_r-\lambda)_+]$  $\left(U\Sigma V^T\right.$ denotes the singular value decomposition of $\left.X\right)$, $\Sigma=diag[\sigma_1,\ldots,\,\sigma_r]$, $t_+=\max(t,\,0)$, $r=rank(X)$, $\|\cdot\|_F$ is the Frobenius norm $\left(\|A\|_F=\sqrt{\sum\limits_{i=1}^n\sum\limits_{j=1}^m|a_{ij}|^2}\right)$ and
 $$P_\Omega(X)(i,\,j)=\left\{
                         \begin{array}{ll}
                           X_{i,j}, & \hbox{if}\quad (i,\,j)\in\Omega \\
                           0, & \hbox{if}\quad (i,\,j)\notin\Omega
                         \end{array}.
                       \right.
 $$

\begin{algorithm}
\caption{- Soft-Impute \cite{journal:Mazumder}}
\begin{itemize}
  \item Input: matrix with missing data, $D\in\mathbb{R}^{n\times n}$, and a tolerance $\epsilon>0$.
\end{itemize}
\begin{enumerate}
  \item Inicialize $X^{old}=0$.
  \item Choose a decreasing sequence of scalars: $\lambda_1,\,\lambda_2,\ldots,\,\lambda_K$.
  \item For $i=1$ to $K$ do:
  \begin{enumerate}
    \item Calcule $X^{new}\rightarrow S_{\lambda_i}(P_\Omega(D)+P_\Omega^\perp(X^{old}))$.
    \item If $\dfrac{\|X^{new}-X^{old}\|_F^2}{\|X^{old}\|_F^2}<\epsilon$, exit.
    \item Do $X^{old}\leftarrow X^{new}$.
  \end{enumerate}
\end{enumerate}
\begin{itemize}
    \item Output: $X^{new}$.
\end{itemize}
\end{algorithm}

We propose a more specific formulation focused on Euclidean Distance Matrices. The idea is to exploit the properties of an EDM (the fact of its rank is at most $k+2$) to obtain some additional advantages in the optimization problem. We want to estimate the matrix $\hat{D}$ by solving the following optimization problem:
\begin{equation}
\label{fixed-rank}
    \begin{array}{ll}
      \emph{minimize } & \displaystyle \sum\limits_{(i,\,j)\in\Omega}(d_{ij}-\hat{d}_{ij})^2 \\
      \emph{subject to}\,\,\,  & \mbox{rank}(\hat{D}) = k+2. \\
    \end{array}
\end{equation}

Since there is not a straightforward formula for $\mbox{rank}(\hat{D})$ as a function of the variables $\hat{d}_{ij}$, we propose a heuristic approach to solve (\ref{fixed-rank}) based on a modification in the Soft-impute Algorithm. The resulting method is called the \textit{Fixed-Rank Soft-Impute} and it is described in the next section.

\subsection{Fixed-Rank Matrix Completion}

There are two main difficulties to apply Algorithm $1$ to solve the formulation (\ref{fixed-rank}) -- actually to solve any problem where the rank of $D$ it is a prior information. The first difficult is that the parameter $\lambda$ is defined outside of the loop, implying that it is insensitive to the convergence that occurs inside the loop. The second difficulty is that there is no specific details about the choice of the parameter $\lambda$. In \cite{journal:Mazumder}, the authors mention a \textit{warm starts}, but without presenting any concrete way of obtaining the values of $\lambda$. In our experiments, we realized that convergence of Algorithm $1$ seems to be very sensitive on the choice of $\lambda$ which motivate us to propose a new way of approach.

The new way to calculate the regulator parameter was also thought out to overcome the second difficult on the use of  Algorithm $1$ to solve (\ref{fixed-rank}). Since now we know the rank of the target matrix, let say rank$(X)=r$, we propose to compute the values of $\lambda$, in each iteration, as
\begin{equation}
\label{new_lambda}
\lambda=\beta\sigma_{r+1},
\end{equation}
where $\beta\in(0,1)$ and $\sigma_{r+1}$ is the $(r+1)$th singular value of the matrix $P_\Omega(D)+P_\Omega^\perp(X^{old})$, where $X^{old}$ is the most recent approximation.

\indent The idea is to link the parameter $\lambda$ with the most recent approximation of $X$ and leave its calculation automatically. Since we wish to converge to a matrix $X$, such that rank$(X) = r$, we should have only $r$ singular values greater than zero. Thus, the singular value $\sigma_{r+1}$ give us a valuable information about how far we are from the solution and, therefore, we can calibrate the step size of the next iteration.

The only exogenous parameter we have now is the value of $\beta \in (0,1)$, that may be adjusted to control the rate of decreasing the singular values in each iteration of the method. The big is $\beta$ the fast is the way of decreasing the singular values in each iteration. Some matrices are very sensitive on the rate of decreasing of the less significant singular values and Algorithm $2$ may lead to a rank $r$ matrix without achieving the tolerance $\epsilon$. In this case, we suggest to try a smaller value for $\beta$ and so running a more slowly decreasing.


We summarize our approach in the following algorithm:

\begin{algorithm}[H]
\caption{- Fixed-Rank Soft-Impute}
\begin{itemize}
  \item Inputs: a matrix with missing data, $D\in\mathbb{R}^{n\times n}$, the rank $r$ of the target matrix $X$ and a tolerance $\epsilon>0$.
\end{itemize}
\begin{enumerate}
  \item Inicialize $X^{old}=0$.
  \item Choose a initial value to $\lambda$ ($\lambda=\lambda_0$).
  \item For $i=1$ to $K$ do:
  \begin{enumerate}
    \item Calcule $X^{new}\rightarrow S_{\lambda}(P_\Omega(D)+P_\Omega^\perp(X^{old}))$.
    \item Set $\lambda=\beta\sigma_{r+1}$, with $\beta\in(0,1)$.
    \item If $\dfrac{\|X^{new}-X^{old}\|_F^2}{\|X^{old}\|_F^2}<\epsilon$, exit.
    \item Do $X^{old}\leftarrow X^{new}$.
  \end{enumerate}
\end{enumerate}
\begin{itemize}
    \item Output: $X^{new}$.
\end{itemize}
\end{algorithm}


We remark that there is no additional cost to obtain $\lambda$, since the SVD is calculated in step 3(a). The apparently naive modification on the calculation of $\lambda$ has been proved to be highly efficient, allowing us to recover incomplete EDMs with high precision, even in the case where less than $5\%$ of the positions of $D$ are known. In the next section, we present some results of computational experiments we have performed using Algorithm 2.


The tests performed were based on a random generation of distance matrices and random deletion of a percentage $p$ of their elements. The original matrix $D$ was maintained for comparison purposes. We used the relative error, defined by $er = \dfrac{\|D -\hat{D}\|_F^2}{\|D\|_F^2}$, where $||A||_F$ represents the Frobenius norm of matrix $A$. We also compute the maximum error $(max\_err)$ between two corespondent elements, i.e.  $max\_err  = max|d_{ij}-\hat{d}_{ij}|$, which is a very rigorous criteria of convergence.

\section{Numerical results}

We have implemented the methods in Matlab language and performed the tests using a microcomputer with Windows-$64$ bits, Intel Core i-$5$, $2430M$ CPU, $2.40$GHz and $4$GB of RAM. \\
\indent We considered two classes of matrices. In order to test Algorithm 2 in a general context, we generated $n\times n$ matrices of rank $r$ according to the form $M=M_LM^*_R$, where $M_L$ and $M_R$ are $n\times r$ independent matrices, having i.i.d. Gaussian entries, as done in \cite{journal:Cai}. Another class is the main interest of this paper. We generated random points in $\mathbb{R}^k$ and then calculated the correspondent EMD. In both cases, we deleted a percentage $p$ of his entries, uniformly at random. All results presented are averages of $10$ simulations. \\
\indent In order to set a good choice for the parameter $\beta$, we performed a set of warm up simulations varying $\beta$ in the interval $(0,\,1)$ and then choosing the most appropriated value for each type of test. Figure \ref{beta} shows the results of these simulations for EDMs, suggesting that, for this case, we may consider $\beta \approx 0.8$.
\begin{figure}[H]
  \centering
  \includegraphics[scale=.3 ]{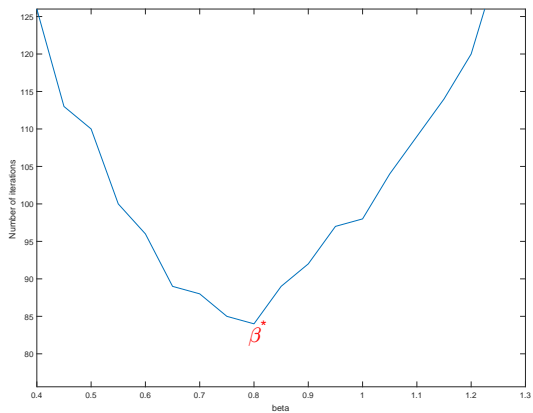}
  \caption{{\scriptsize{Number of iterations (average) versus beta. We made $100$ simulations, using $100\times 100$ EDMs associated to points in $\mathbb{R}^2$, with $40\%$ of missing data.}}}
  \label{beta}
\end{figure}

In Table \ref{tab3}, we present a comparison between results obtained with Algorithm $2$, for the case of random matrices of form $M=M_LM^*_R$, and results presented in \cite{journal:Cai} for the same size of matrices and percentage of deletion.\\


\begin{table}[H]{\scriptsize{
\centering
\begin{tabular}{|c|c|c|c|c|c|c|c|c|}
\hline
\multicolumn{3}{|c|}{} & \multicolumn{6}{c|}{Results}\tabularnewline
\hline
\multicolumn{3}{|c|}{Unknow matrix
} & \multicolumn{3}{c|}{Results by \cite{journal:Cai} } & \multicolumn{3}{c|}{Our approach}\tabularnewline
\hline
size ($n\times n$) & rank ($r$) & $|\Omega|/n^{2}$ & time ($s$) & \# iter  & relative error & time ($s$) & \# iter  & relative error\tabularnewline
\hline
\multirow{3}{*}{$1000\times1000$} & $10$ & $0.12$ & $23$ & $117$ & $1.64\times10^{-4}$ & $226$ & $386$ & $9.95\times10^{-5}$\tabularnewline
\cline{2-9}
 & $50$ & $0.39$ & $193$ & $114$ & $1.59\times10^{-4}$ & $73.2$ & $103$ & $9.93\times10^{-5}$\tabularnewline
\cline{2-9}
 & $100$ & $0.57$ & $501$ & $129$ & $1.68\times10^{-4}$ & $129$ & $79$ & $9.68\times10^{-5}$\tabularnewline
\hline
\end{tabular}
\caption{{\scriptsize{Comparison between our approach and Cai et al in \cite{journal:Cai}, for a  $1000\times1000$ matrix.}}}
\label{tab3}
}}
\end{table}


In Table \ref{tab5}, we show the advantage of our approach. With the same parameters  used in Table \ref{tab3}, we set the tolerance $\epsilon = 10^{-10}$. We can see that, in the worst case, we were able to recover $D$ with $7$ correct decimal places ($max\_error \leq 10^{-8}$), showing the high accuracy of our approach.

\begin{table}[H]{\scriptsize{
  \centering
  \begin{tabular}{|c|c|c|c|c|c|c|}
    \hline
    \multicolumn{3}{|c|}{$Unknow\ matrix$} & \multicolumn{4}{|c|}{$Results$}\tabularnewline
    \hline
    \hline
    $size\ (n\times n)$ & $rank\ (r)$ & $|\Omega|/n^{2}$ & $relative\ error$ & $max\ error$ & $\#\ iter$ & $time\ (s)$\tabularnewline
    \hline
    \multirow{3}{*}{$1000\times1000$} & $10$ & $0.12$ & $9.83\times10^{-11}$ & $1.18\times10^{-9}$ & $1084$ & $1003.2$\tabularnewline
    \cline{2-7}
     & $50$ & $0.39$ & $9.72\times10^{-11}$ & $8.16\times10^{-9}$ & $244$ & $238.2$\tabularnewline
    \cline{2-7}
     & $100$ & $0.57$ & $9.35\times10^{-11}$ & $1.94\times10^{-8}$ & $201$ & $192.6$\tabularnewline
    \hline
\end{tabular}
  \caption{{\scriptsize{$1000\times1000$ matrix completion performed by our approach, setting tolerance equal to $10^{-10}$.}}}
  \label{tab5}
  }}
\end{table}

Now, we focus on EDM's. We generated points in $\mathbb{R}^k$, according to a uniform distribution  and then compute the corresponding EDM $D$. For each experiment we deleted a percentage $p$ of the entries of $D$, uniformly at random, and we applied Algorithm $2$ to recover the missing data.

Table \ref{tab2} shows numerical results of tests performed with EMD's obtained from random generation. We varied the number of points generated, the space dimension, the percentage of deletion, and we fixed the tolerance (relative error) equal to $10^{-8}$.\\

\begin{table}[H]{\scriptsize{
\centering

\begin{tabular}{|c|c|c|c|c|c|}
\hline
\# Points (n) & Dimension (d) & Rank (r) & Deletion (\%) & \# Iteration  & Maximum error\tabularnewline
\hline
\hline
$500$ & $10$ & $12$ & $50$ & $61$ & $2.76\times10^{-7}$\tabularnewline
\hline
$1000$ & $3$ & $5$ & $70$ & $82$ & $7.11\times10^{-8}$\tabularnewline
\hline
$2000$ & $10$ & $12$ & $70$ & $86$ & $2.45\times10^{-7}$\tabularnewline
\hline
$5000$ & $50$ & $52$ & $50$ & $44$ & $1.27\times10^{-6}$\tabularnewline
\hline
$5000$ & $200$ & $202$ & $50$ & $74$ & $8.32\times10^{-6}$\tabularnewline
\hline
$5000$ & $50$ & $52$ & $80$ & $193$ & $1.61\times10^{-6}$\tabularnewline
\hline
$5000$ & $3$ & $5$ & $90$ & $201$ & $1.24\times10^{-4}$\tabularnewline
\hline
$10000$ & $100$ & $102$ & $80$ & $187$ & $3.40\times10^{-6}$\tabularnewline
\hline
\end{tabular}

\caption{{\scriptsize{Performance of Algorithm $2$ varying some parameters such as the number of points generated, the space dimension, and the percentage of random deletion. For every experiment we fixed the tolerance (relative error) equal to $10^{-8}$.}}}
\label{tab2}
}}
\end{table}

According to these results, we can see that Algorithm $2$ was able to recover missing data with very high accuracy. Additionally, Table \ref{tab2} shows us that we can recover matrices with not so small rank. We recovered, for instance, matrices with rank equal to $100$ and $200$ with a very good accuracy. \\

For the results presented in Table \ref{tab6}, we fixed $n = 1000$ points in dimension $r = 8$ and we varied the percentage of deletion of data: $p \in \{10, 20,\ldots, 90\}$.

\begin{table}[H]
  \centering
  {\scriptsize{
  \begin{tabular}{|c|c|c|c|c|}
\hline
\multirow{2}{*}{Deletion (\%)} & \multicolumn{4}{c|}{Computational results}\tabularnewline
\cline{2-5}
 & $time\ (s)$ & $\#\ iter$ & $relative\ error$ & $max\ error$\tabularnewline
\hline
$10$ & $10.2$ &$24$ & $5.73\times10^{-13}$ & $7.67\times10^{-12}$\tabularnewline
\hline
$20$ & $9$ & $35$ & $6.04\times10^{-13}$ & $1.31\times10^{-11}$\tabularnewline
\hline
$30$ & $17.4$ & $43$ & $5.04\times10^{-13}$ & $7.22\times10^{-12}$\tabularnewline
\hline
$40$ & $18.6$ & $58$ & $6.79\times10^{-13}$ & $8.35\times10^{-12}$\tabularnewline
\hline
$50$ & $27$ & $75$ & $8.53\times10^{-13}$ & $9.43\times10^{-12}$\tabularnewline
\hline
$60$ & $38.4$ & $110$ & $9.08\times10^{-13}$ & $1.72\times10^{-11}$\tabularnewline
\hline
$70$ & $69$ & $184$ & $9.18\times10^{-13}$ & $1.32\times10^{-11}$\tabularnewline
\hline
$80$ & $170.4$ & $430$ & $9.42\times10^{-13}$ & $2.05\times10^{-11}$\tabularnewline
\hline
$90$ & $367.8$ & $1000$ & $2.91\times10^{-07}$ & $3.30\times10^{-06}$\tabularnewline
\hline
\end{tabular}
}}
  \caption{{\scriptsize{Performance of Fixed Rank Soft Impute. Results of a simulation with an EDM associated to 1000 points in $\mathbb{R}^8$, with random deletion $p \in \{10, 20,\ldots, 90\}$, a maximum number of iteration of 1000 and tolerance equal to $10^{-12}$.}}}
  \label{tab6}
\end{table}

As expected, as the percentage of the missing data increases, the algorithm needs more iterations to converge. Even so, in all cases we were able to achieve at error maximum up to $10^{-6}$ with less than 1000 iterations in a few minutes of calculation.\\

We also varied the size of EDMs, using points in $\mathbb{R}^5$ and deleting $p=70\%$ of the data, to see how our approach behaves. We summarized the results in Table \ref{tab7}.

\begin{table}[H]
  \centering
  {\scriptsize{
  \begin{tabular}{|c|c|c|c|c|}
\hline
\multirow{2}{*}{$\# points$ ($n$)} & \multicolumn{4}{c|}{Computational results}\tabularnewline
\cline{2-5}
 & $time\ (s)$ & $\#\ iter$ & $relative\ error$ & $max\ error$\tabularnewline
\hline
$200$ & $10.8$ & $473$ & $9.72\times10^{-09}$ & $5.61\times10^{-08}$\tabularnewline
\hline
$500$ & $16.2$ & $165$ & $9.33\times10^{-09}$ & $8.17\times10^{-08}$\tabularnewline
\hline
$1000$ & $29.4$ & $108$ & $9.13\times10^{-09}$ & $1.05\times10^{-07}$\tabularnewline
\hline
$2000$ & $149.4$ & $83$ & $8.16\times10^{-09}$ & $8.46\times10^{-08}$\tabularnewline
\hline
\end{tabular}
}}
  \caption{{\scriptsize{Performance of Fixed Rank Soft Impute. Results of a simulation with a EDM associated to $n$ points $(n\in\{200,\ 500,\ 1000,\ 2000\})$ in $\mathbb{R}^5$, with random deletion of $70\%$ of data, a maximum number of iteration of 1000 and tolerance equal to $10^{-08}$.}}}
  \label{tab7}
\end{table}

We can note that, for a fixed percentage of missing data, the number of required iterations decreases when the size of the EMD grows. As we can see, in all cases tested, our approach has been proved to be an efficient way to recover missing data in EDMs, outperforming other competitive state-of-art technique. Finally, we have performed some examples with EDMs related to 5000 random points in $\mathbb{R}^3$ and considered $p=0.98$ of missing data. In these cases, our approach was able to recover the exact matrix with absolute error smaller than $10^{-7}$.

\section{Conclusions}

We proposed the application of low-rank matrix completion approach to estimate missing data in Euclidean distance matrices (EDMs). We presented a specific formulation of the matrix completion and proposed an algorithm based on singular value decomposition. Computational results indicate that our approach is efficient to estimate missing positions in matrices whose rank is known a priori. We were able to recover matrices from a few percentage of entries, with very  high accuracy. The method presented here can also be used for other applications in matrix completion when the rank of the target matrix is know.

\section*{Acknowledgements}

The authors would like to thank the Brazilian research agencies CAPES, CNPq, and FAPESP for their financial support.


\end{document}